\documentclass[a4paper]{amsart}

\usepackage{amsmath,amscd,amsthm,amsfonts,amssymb,esint,mathtools}
\usepackage{fullpage}
\usepackage{booktabs,tabularx,array}
\usepackage{enumitem}
\usepackage{microtype}
\usepackage[colorlinks=true,linkcolor=blue,citecolor=blue,urlcolor=blue]{hyperref}
\usepackage[nameinlink,noabbrev]{cleveref}

\allowdisplaybreaks
\setlist{nosep}
\newcolumntype{Y}{>{\raggedright\arraybackslash}X}

\newtheorem{theorem}{Theorem}[section]
\newtheorem{proposition}[theorem]{Proposition}
\newtheorem{lemma}[theorem]{Lemma}
\newtheorem{corollary}[theorem]{Corollary}
\theoremstyle{definition}
\theoremstyle{remark}
\newtheorem{remark}[theorem]{Remark}

\crefname{equation}{Equation}{Equations}
\crefname{theorem}{Theorem}{Theorems}
\crefname{proposition}{Proposition}{Propositions}
\crefname{lemma}{Lemma}{Lemmas}
\crefname{corollary}{Corollary}{Corollaries}
\crefname{remark}{Remark}{Remarks}
\crefname{section}{Section}{Sections}
\crefname{subsection}{Subsection}{Subsections}
\crefname{table}{Table}{Tables}
\crefname{lstlisting}{Listing}{Listings}

\newcommand{\F}{\mathbf F}
\newcommand{\Q}{\mathbf Q}
\newcommand{\Z}{\mathbf Z}
\newcommand{\rad}{\operatorname{rad}}
\newcommand{\Leg}[2]{\left(\frac{#1}{#2}\right)}

\usepackage{listings}
\usepackage{url}
\begin{document}
\title{On Sun's Conjectures for Truncated Jacobi-Symbol Determinants via Supersingular Elliptic Curves}
\author{Guo Li}

\address{School of Mathematical Sciences, University of Chinese Academy of Sciences, No. 19A Yuquan Road,  Beijing 100049, China}
\email{liguo22@mails.ucas.ac.cn}

\author{Xiaojun Yan}
\address{Academy of Mathematics and Systems Science,
Chinese Academy of Sciences,
Beijing
100190, China.}
\email{xjyan95@amss.ac.cn}

\begin{abstract}
We study the truncated Jacobi-symbol determinants 
$$\{c,d\}_n=\det\!\left[\left(\frac{j^2+cjk+dk^2}{n}\right)\right]_{2\le j,k\le n-2}$$ proposed by Zhi-Wei Sun in \cite{ZWSun} and prove Conjectures 5.1(i), 5.2, 5.3, 5.4, 5.5, 5.6(i), 5.7, 5.8, a case of 5.6(ii) of \cite{ZWSun},  Conjecture 4.8(i) of \cite{Sun2019} and some strengthened forms. All results follow from a single unified approach:  for a prime $p$, we diagonalize the nonzero-residue matrix indexed by $\F_p^\times$ and reduce the vanishing of determinants to the supersingular reduction of certain CM elliptic curves.  The same framework extends naturally to further families of parameters, suggesting a general mechanism behind identities of this type.
\end{abstract}


\maketitle

\tableofcontents

\section{Introduction}

Determinants whose entries are Legendre or Jacobi symbols occupy a rich intersection of arithmetic and linear algebra, where closed-form evaluations often encode deep number-theoretic constraints. For an odd prime $p$, let $\Leg{\cdot}{p}$ denote the Legendre symbol. Carlitz \cite{Ca59} proved  $ \det\left[\Leg{j-i}{p}\right]_{1\le i,j\le p-1}=p^{(p-3)/2}.$ Chapman evaluated several Hankel-type Legendre-symbol determinants
\cite{Ch04} and formulated the famous  ``evil determinant" conjecture \cite{Ch12} concerning the truncated  determinant 
\[
\det\left[\Leg{j-i}{p}\right]_{1\le i,j\le \frac{p+1}{2}},
\]
which was subsequently proved by Vsemirnov \cite{Vs12,Vs13}.  
Sun \cite{Sun2019} broadened the class of determinants under investigation by replacing the linear form $j-i$ with a binary quadratic form $j^2+cjk+dk^2$. Krachun, Petrov, Sun, and Vsemirnov \cite{KPSV} subsequently established vanishing results for the full and nonzero-residue determinants (denoted below by $[c,d]_n$ and $(c,d)_n$, respectively), confirming (ii) and part of (iv) in \cite[Conjecture~4.8]{Sun2019}. Their proofs use identities for character sums and explicit constructions of kernel vectors.
Wu \cite{Wu2021} related these determinants to Frobenius traces of elliptic curves. Among other results, he obtained a vanishing criterion for $[c,d]_p$ in terms of the central coefficient of $(dX^2+cX+1)^{\frac{p-1}{2}}$ and applied Elkies' theorem to prove that $[c,d]_p=0$ for infinitely many primes $p$ for each fixed pair of integers $(c,d)$ satisfying $d(c^2-4d)\ne0$.

Sun's recent paper \cite{ZWSun} conjectures further congruences and vanishing statements for the \emph{truncated} determinant $\{c,d\}_n$, obtained by deleting the three residue classes $0,1,-1$.  The truncation is the central difficulty: it destroys the multiplicative symmetry of the matrix indexed by $\F_p^\times$. A kernel vector of the nonzero-residue matrix need not vanish at the deleted indices $1$ and $-1$.  Our strategy is to temporarily reinstate this symmetry, diagonalize the resulting matrix and then impose the two boundary conditions---this reduces each conjecture to showing that  specific coefficients of $(X^2+cX+d)^{\frac{p-1}{2}}$ vanish. 
Beyond the central coefficient used by Wu, we apply Z.-H. Sun's results on Legendre polynomials \cite{SunLegendreII,Sun2013} to identify the one-third and one-quarter coefficients with Hasse invariants of explicitly given elliptic curves; computing their $j$-invariants reveals complex multiplication, and Deuring's reduction theorem converts the inertness of $p$ in the CM field into supersingularity, which forces the coefficient to vanish. Together with direct calculations for low-degree coefficients, these steps yield a unified proof of all the results stated in Theorems~\ref{thm:main} and~\ref{thm:oldSUN} below. For extensions to further parameter families, see Remark~\ref{remarkunify}.
\subsection{Notation and main results}

Let $n>1$ be an odd integer, and  $c,d\in\Z$.  Following
\cite[Equations (1.21)--(1.22)]{Sun2019} and \cite[Equation (5.1)]{ZWSun},
write
\begin{align*}
 \mathcal M_n^{\mathrm{full}}(c,d)&=
 \left[\Leg{j^2+cjk+dk^2}{n}\right]_{0\le j,k\le n-1},
 \quad  [c,d]_n=\det\mathcal M_n^{\mathrm{full}}(c,d),\\
 \mathcal M_n^\times(c,d)&=
 \left[\Leg{j^2+cjk+dk^2}{n}\right]_{1\le j,k\le n-1},
 \quad (c,d)_n=\det\mathcal M_n^\times(c,d),\\
 \mathcal M_n^{\mathrm{trun}}(c,d)&=
 \left[\Leg{j^2+cjk+dk^2}{n}\right]_{2\le j,k\le n-2},
 \quad \{c,d\}_n=\det\mathcal M_n^{\mathrm{trun}}(c,d).
\end{align*}
Here $\Leg{\cdot}{n}$ denotes the Jacobi symbol. Our first result resolves almost all of Sun's conjectures regarding $\{c,d\}_n$ that were proposed in \cite{ZWSun}, by employing the unified approach proposed in \cref{sec:strategy}.

\begin{theorem}[Main theorem]\label{thm:main}
Let $n>1$ be an odd integer and $p$ a prime.
\begin{enumerate}[label=\textup{(\alph*)}]
 \item If $n\equiv1\pmod4$ is not a sum of two squares, then
 $\{3,2\}_n=0$.
 \item If $p\equiv13,19\pmod {24}$, then $\{2,2\}_p=0$; if
 $p\equiv17,23\pmod {24}$, then $\{2,2\}_p\equiv0\pmod p$.
 \item If $n\equiv5\pmod8$, then
 $\{4,2\}_n=\{8,8\}_n=0$.  If $n\equiv5\pmod {12}$, then
 $\{3,3\}_n=0$.
 \item If $n\equiv1\pmod4$ and $\Leg n7=-1$, then
 $\{42,-7\}_n=\{21,112\}_n=0$.
 \item If $n>3$, then $n\mid\{2,3\}_n$.  Moreover,
 $n^2\mid\{2,3\}_n$ whenever $n\not\equiv\pm1\pmod {12}$; in fact,
 this divisibility holds for every composite $n>3$.  For every odd $n>7$, one has
 $n\mid\{6,15\}_n$.
 \item If $n\equiv13,17\pmod {20}$ is a sum of two squares, then
 $\{5,5\}_n=0$.  If $n\equiv11,19\pmod {20}$, then
 $\Leg{\{5,5\}_n}n=0$.
 \item If $n\equiv5\pmod {12}$ is a sum of two squares, then
 $\{10,9\}_n=0$.  If $p\equiv11\pmod {12}$, then
 $\{10,9\}_p\equiv0\pmod p$.
 \item If $n\equiv13,17\pmod {24}$ is a sum of two squares, then
 $\{8,18\}_n=0$.  If $p\equiv19\pmod {24}$, then
 $\{8,18\}_p\equiv0\pmod {p^2}$, while if
 $p\equiv23\pmod {24}$, then $\{8,18\}_p\equiv0\pmod p$.
\end{enumerate}
\end{theorem}
\begin{remark}
    Parts (a)--(h) correspond respectively to Conjectures 5.1(i), 5.2, 5.3, 5.4, 5.5, 5.6(i), 5.7, and 5.8 of \cite{ZWSun}. Part (f) also covers the $11,19\pmod{20}$ cases of Conjecture 5.6(ii).
\end{remark}

For the full nonzero-residue determinant $(c,d)_n$ defined above, the following additional results hold.

\begin{theorem}[{\cite[Conjecture 4.8(i)]{Sun2019}}]\label{thm:oldSUN}
 For every odd integer $n>3$, $n^2\mid(2,3)_n$; for every odd integer $n>5$, $n^2\mid(6,15)_n$.
\end{theorem}

\begin{remark}\label{remarkunify}[A general mechanism]
    The unified approach extends well beyond the cases listed above.  See \cref{rem:gene} for a discussion of the underlying generation mechanism and further families of examples.
\end{remark}



\subsection{Proof strategy}\label{sec:strategy}

The three deleted indices $0,1,-1$ are the principal obstacle: excising them destroys the multiplicative symmetry of the matrix indexed by $\F_p^\times$, so that standard group-determinant techniques no longer apply directly. We therefore temporarily restore  $\mathcal M_p^\times(c,d)$.  This matrix is diagonalized by  vectors $(1^s,2^s,\dots,(p-1)^s)^{\mathsf T}$ over $\F_p$, and the eigenvalues are directly related to the coefficients in the expansion of $q(X)^{\frac{p-1}{2}}$ where $q(X)=X^2+cX+d$. 

This observation reduces the determinant calculation to the problem of showing that certain coefficients vanish, which arises from the following three distinct sources. 
\begin{itemize}
    \item For cases (a), (c), and (d) of \cref{thm:main}, the relevant term is the central coefficient (i.e., the coefficient of degree $\frac{p-1}{2}$), which corresponds directly to the Hasse invariant of certain elliptic curves.
    \item For case (e) of \cref{thm:main} and \cref{thm:oldSUN}, the relevant coefficients are of degrees two and three, and their vanishing can be verified via direct expansion.
    \item  The remaining cases of \cref{thm:main} involve coefficient identities at one-third (for case (b)) and one-quarter (for case (f), (g) and (h)) of the  maximal degree $p-1$. In these instances, a crucial ingredient in our approach is to apply Z.-H. Sun's results \cite{Sun2013,SunLegendreII}, which elegantly reduce their vanishing to the Hasse invariants of explicitly defined elliptic curves as well.
\end{itemize}
Having reduced the corresponding scenarios to the study of Hasse invariants, we then compute the $j$-invariants of these curves to identify the relevant CM fields. From there, Deuring's theorem translates the inertness of $p$ into supersingularity, which consequently establishes the required vanishing of these coefficients.

The passage from coefficient vanishing back to the truncated determinant is independent of the arithmetic input. Specifically, a pair of reciprocal zero coefficients yields a vector over $\F_p$ that vanishes at $1$ and $-1$, thereby naturally providing a kernel vector for the truncated matrix. For conjectures asserting exact equality over $\Z$, rather than merely divisibility by $p$, we replace this vector with a character-coset indicator function and apply an absolute-value bound to deduce that every row-vector product evaluates strictly to zero in $\Z$. For composite moduli, the Chinese Remainder Theorem lifts these local kernels to the full residue matrix, where some exceptional small primes are resolved via explicit calculations.

\subsection{Acknowledgments}
The authors thank Zhihan Zhong for the exchanges and discussions during the writing process. They also thank the anonymous referees for the valuable suggestions.
The key observation linking the determinant problem to elliptic curves originated from the proof of case (b) of \cref{thm:main} via the one-third coefficient. GPT-5.6 Sol helped us quickly test the feasibility of this strategy on other conjectures of Sun. The authors assume full responsibility for the correctness, originality, and final presentation of all mathematical statements in this paper.

\section{Basic Observations}

We first identify the eigenvalues of $\mathcal M_p^\times(c,d)$ over $\F_p$.  Subsequent results describe the effect of deleting the boundary coordinates and the process of lifting these properties to a composite modulus. These observations will be used repeatedly later.

Let $p$ be an odd prime such that $p\nmid d$. The coefficients employed in the sequel are defined by
\begin{equation}\label{eq:q-expansion}
 q(X)=X^2+cX+d,
 \qquad q(X)^{\frac{p-1}{2}}=\sum_{s=0}^{p-1}a_sX^s
 \quad\text{in }\F_p[X].
\end{equation}
\begin{lemma}\label[lemma]{lem:spectrum}
For $0\le s\le p-2$, let  $\vec w_s=(1^s,2^s,\ldots ,(p-1)^s)^{\mathsf T}.$ Then
these vectors form a basis of $\F_p^{p-1}$.  Moreover, the following holds:
\begin{enumerate}
    \item We have
    \begin{equation}\label{eq:eigen}
    \mathcal M_p^\times(c,d)\vec w_0=-(a_0+a_{p-1})\vec w_0,
    \qquad
    \mathcal M_p^\times(c,d)\vec w_s=-a_s\vec w_s
    \quad(1\le s\le p-2).
    \end{equation}

\item $a_{p-1-s}=d^{s-\frac{p-1}{2}}a_s $ for $0\le s\le p-1$,  $a_0=\Leg dp$, and  $a_{p-1}=1$.
\item If $a_s=0$ for some $1\le s\le p-2$,
then the vector obtained by appending a $0$ at the zeroth coordinate to $\vec w_s$ belongs to $\ker_{\F_p}\mathcal M_p^{\mathrm{full}}(c,d)$.
\end{enumerate}
\end{lemma}

\begin{proof}
The non-vanishing of the Vandermonde determinant ensures that $\{\vec w_s\}_{0\le s\le p-2}$ constitutes a basis. For $j\in\F_p^\times$, substituting $x=j/k$ yields
\begin{align*}
 (\mathcal M_p^\times(c,d)\vec w_s)_j
 =\sum_{k\in\F_p^\times}q(j/k)^{\frac{p-1}{2}} k^s
 =\sum_{x\in\F_p^\times}q(x)^{\frac{p-1}{2}}(j/x)^s
 =j^s\sum_{u=0}^{p-1}a_u\sum_{x\in\F_p^\times}x^{u-s}.
\end{align*}
Therefore, the power-sum identity $ \sum_{x\in\F_p^\times}x^e=
 \begin{cases}
  -1,&p-1\mid e,\\
  0,&p-1\nmid e
 \end{cases}\quad\text{in }\F_p $ yields \cref{eq:eigen}. Furthermore, the algebraic relation $q\!\left(\frac dX\right)=dX^{-2}q(X)$ implies that $ \sum_s a_sd^sX^{-s}=d^{\frac{p-1}{2}} \sum_u a_uX^{u-(p-1)}.$ Comparing the coefficients on both sides then gives
$a_{p-1-s}=d^{s-\frac{p-1}{2}}a_s$. Finally, for $1\le s\le p-2$, we have $\sum_{k\ne0}\Leg{d k^2}{p}k^s
=\Leg dp\sum_{k\ne0}k^s=0,$
which justifies the assertion regarding the extension by zero.
\end{proof}

The information lost during the truncation of the matrix is captured by evaluating the kernel vectors at the removed nonzero indices.

\begin{proposition}\label[proposition]{prop:boundary}
Let $W$ be a subspace of $\ker_{\F_p}\mathcal M_p^\times(c,d)$. Then we have:
\begin{enumerate}
    \item The coordinate restriction maps the kernel of the evaluation map
\[
 \operatorname{ev}:W\longrightarrow\F_p^2,
 \qquad \vec v\longmapsto(v_1,v_{-1}),
\]
injectively into $\ker\mathcal M_p^{\mathrm{trun}}(c,d)$. Thus
$\dim_{\F_p}\ker\mathcal M_p^{\mathrm{trun}}(c,d)
\ge \dim W-\operatorname{rank}(\operatorname{ev}).$

\item If $1\le s<p-1-s\le p-2$ and $a_s=a_{p-1-s}=0$, then
$0\ne \vec w_s-\vec w_{p-1-s}\in \ker_{\F_p}\mathcal M_p^{\mathrm{trun}}(c,d)$. More generally, let $S\subseteq\{0,1,\ldots,p-2\}$ be such that
$a_s=0$ for every $s\in S\setminus\{0\}$ and
$a_0+a_{p-1}=0$ whenever $0\in S$.  For
$W=\operatorname{span}\{\vec w_s:s\in S\}$ one has
\[
 v_p\!\left(\{c,d\}_p\right)
 \ge \dim W-\operatorname{rank}(\operatorname{ev}),
\]
where $v_p(0)=+\infty$.
\end{enumerate}
\end{proposition}

\begin{proof}
The final assertion employs the Smith normal form over $\Z$, while the others are straightforward.
\end{proof}

Although the preceding proposition guarantees the existence of a kernel vector over $\F_p$, establishing exact singularity over $\Z$ requires an integer vector whose row-vector products vanish identically in $\Z$.

\begin{lemma}\label[lemma]{lem:quadratic-boundary}
Suppose that $p\ge5$ is a prime satisfying $p\equiv1\pmod4$ and $p\nmid d(c^2-4d)$. If $\sum_{x\in\F_p}\Leg{xq(x)}p=0 $ and $ \Leg dp=-1$, then the vector $\vec v=(v_k)_{2\le k\le p-2}$ defined by $v_k=\Leg kp-1$ is a nonzero kernel vector of $\mathcal M_p^{\mathrm{trun}}(c,d)$ over $\Z$. Consequently, $\{c,d\}_p=0$.
\end{lemma}

\begin{proof}
The number of pairs $(x,y)\in\F_p^2$ satisfying $y^2=q(x)$ is $p+\sum_x\Leg{q(x)}p$. On the other hand, this equation is equivalent to
\[
 \left(x+\frac{c}{2}-y\right)\left(x+\frac{c}{2}+y\right)=\frac{c^2-4d}{4}\not\equiv 0\pmod p.
\]
Each choice of the first factor in $\F_p^\times$ uniquely determines the second factor, which in turn uniquely determines $x$ and $y$. Thus, there are exactly $p-1$ solutions. Consequently, we obtain $\sum_{x\in\F_p}\Leg{q(x)}p=-1$, which implies that $\sum_{x\ne0}\Leg{q(x)}p=-1-\Leg dp=0.$ For $j\ne0$, substituting $x=j/k$ yields
\[
 \sum_{k\ne0}\Leg{q(j/k)}p\left[\Leg kp-1\right]
 =\Leg jp\sum_{x\ne0}\Leg{xq(x)}p-\sum_{x\ne0}\Leg{q(x)}{p}=0,
\]
by our hypothesis.  Since $p\equiv1\pmod4$, both $1$ and $-1$ are  quadratic residues modulo $p$, ensuring that the vector vanishes at $k=\pm1$. Finally, the vector is nonzero because $\F_p^\times$ necessarily contains a quadratic non-residue.
\end{proof}

\begin{lemma}\label[lemma]{lem:composite}
Let $n>1$ be odd.
\begin{enumerate}[label=\textup{(\roman*)}]
\item If $n$ is not squarefree, then $\{c,d\}_n=0$ for all integers $c$ and $d$.
\item If $n$ is squarefree, then (up to reordering of indices)
\[
 \mathcal M_n^{\mathrm{full}}(c,d)
 =\bigotimes_{\ell\mid n}\mathcal M_\ell^{\mathrm{full}}(c,d),
\]
where $\otimes$ denotes the Kronecker product. Furthermore, for any prime divisor $p\mid n$,  if $\vec u$ is a nonzero vector in $\ker\mathcal M_p^{\mathrm{full}}(c,d)$ over a field $K$, then there exists a nonzero vector $ \vec z \in K^{n/p}$ such that  $0\ne \vec u \otimes \vec z\in \ker\mathcal M_n^{\mathrm{trun}}(c,d)$ provided that either $n>p$ and $u_0=0$, or $u_0=u_1=u_{-1}=0$.
\item Suppose that $n$ is squarefree. Let $B \subset \Z/n\Z$ be a set of residue classes, and let $b=\vert{}B\vert{}$. For a prime divisor $p$ of $n$, let $h=\dim_{\F_p}\ker\mathcal{M}_p^{\mathrm{full}}(c,d)$. Then the principal submatrix obtained by deleting the indices corresponding to $B$ has nullity at least $h\frac{n}{p}-b$ over $\F_p$.
\end{enumerate}
\end{lemma}

\begin{proof}
For (i), the value of the Jacobi symbol $\Leg{x}{n}$ depends only on the residue class of $x$ modulo $\rad(n)$ (the product of the distinct prime factors of $n$). Thus, we have $\operatorname{rank} \mathcal M_n^{\mathrm{trun}}(c,d) \le \rad(n)$. If $n$ is divisible by the square of a prime, then $\rad(n)\le\frac{n}{3}<n-3$. As the truncated matrix has size $n-3$, its determinant necessarily vanishes.

For (ii), the decomposition follows directly from the Chinese Remainder Theorem.
Assuming $n>p$, we have $n/p\ge3$. We may therefore select a nonzero vector $\vec z \in K^{n/p}$ that vanishes at the residue classes $1$ and $-1$ modulo $n/p$. The tensor product $\vec u\otimes \vec z$ then yields the desired vector. The remaining case is immediate.

For (iii), the tensor decomposition implies that $\operatorname{rank}_{\F_p}\mathcal M_n^{\mathrm{full}}(c,d)\le(p-h)\frac{n}{p}$. Since the retained principal submatrix has size $n-b$, its nullity over $\F_p$ is at least $(n-b)-(p-h)\frac{n}{p}=h\frac{n}{p}-b$.
\end{proof}

\section{Vanishing of the Required Coefficients}

The following theorem is crucial for establishing  the vanishing of specific coefficients.

\begin{theorem}[Deuring's reduction theorem]\label{thm:deuring}
Let $\mathcal{E}$ be an elliptic curve over a number field, with $\mathrm{End}(\mathcal{E}) \cong \mathcal{D}$, where $\mathcal{D}$ is an order in an imaginary quadratic field $K$. Let $\mathcal{P}$ be a place of $\bar{\mathbb{Q}}$ over a prime number $p$, where $\mathcal{E}$ has non-degenerate reduction $E$. The curve $E$ is supersingular if and only if $p$ has one prime of $K$ above it ($p$ is ramified or inert). The curve $E$ is ordinary if and only if $p$ splits completely in $K$.
\end{theorem}
\begin{proof}
    See \cite[Chapter 13, \S 4, Theorem 12]{LangEllipticFunctions} or \cite[Theorem 1.1]{Zaytsev}.
\end{proof}

\subsection{The central coefficient}

The proofs of cases (a), (c), and (d) of \cref{thm:main} rely on the coefficient with index $\frac{p-1}{2}$ in \cref{eq:q-expansion}. This coefficient possesses the most direct geometric interpretation.

Consider
\begin{equation}\label{eq:central-curve}
 E_{c,d}:\quad y^2=x(x^2+cx+d),
\end{equation}
with discriminant and $j$-invariant
\begin{equation}\label{eq:central-invariants}
 \Delta(E_{c,d})=16d^2(c^2-4d),
 \qquad
 j(E_{c,d})=256\frac{(c^2-3d)^3}{d^2(c^2-4d)}.
\end{equation}
Assuming that $p\nmid 2d(c^2-4d)$, $E_{c,d}/\F_p$ defines an elliptic curve over the finite field $\F_p$.

We recall some standard facts concerning elliptic curves over finite fields.
\begin{theorem}
\label[theorem]{rem:elliptic}
Let $p$ be an odd prime, and let $q=p^n$ for some integer $n \ge 1$.
    For an elliptic curve $E:y^2=f(x)$ over a finite field $\F_q$ and any $1<q'\mid q$, denote
\[
 \operatorname{Ha}_{q'}(E)=[x^{q'-1}]f(x)^{\frac{q'-1}{2}},
\]
the coefficient of $x^{q'-1}$ in $f(x)^{\frac{q'-1}{2}}$, and  $t_{q}(E)=q+1-\#E(\F_{q})$.
We have
\begin{enumerate}
    \item 
 $t_q(E)=-\sum_{x\in\F_q}\Leg{f(x)}q\equiv\operatorname{Ha}_q(E)\pmod p$ and $\operatorname{Ha}_p(E)=0$ if and only if $E/\F_q$ is supersingular.
\item If $p\ge5$ and $E/\F_p$ is supersingular, then $t_p(E)=0$.
\end{enumerate}
\end{theorem}
\begin{proof}
   Statements (1) and (2) follow directly from \cite[(V.1.3)]{Silverman}, \cite[Theorem V.4.1]{Silverman} and the Hasse bound (\cite[Theorem V.1.1]{Silverman}) which asserts that $|t_p(E)|\leq 2\sqrt{p}$.
\end{proof}

For \cref{eq:central-curve},
\begin{equation*}
 \operatorname{Ha}_p(E_{c,d})=[x^{p-1}]\{xq(x)\}^{\frac{p-1}{2}}
 =[x^{\frac{p-1}{2}}]q(x)^{\frac{p-1}{2}}=a_{\frac{p-1}{2}}.
\end{equation*}
Consequently, supersingularity implies that the middle eigenvalue of $\mathcal M_p^\times(c,d)$ vanishes on $\F_p^\times$. The values needed in this subsection are displayed in \cref{tab:central-cm} (see \cref{lst:central-cm}).
\begin{table}[ht]
\centering
\caption{Data used for the central coefficient $a_{\frac{p-1}{2}}$.}
\label{tab:central-cm}
\small
\begin{tabular}{@{}c c c c c@{}}
\toprule
\multicolumn{5}{c}{$E_{c,d}:\ y^2=x(x^2+cx+d)$}\\
\midrule
$(c,d)$ & $\Delta(E_{c,d})$ & $j(E_{c,d})$ & CM discriminant $D$ & CM field \\
\midrule
$(3,2)$ & $2^6$ & $1728$ & $-4$ & $\Q(i)$ \\
$(4,2)$ & $2^9$ & $8000$ & $-8$ & $\Q(\sqrt{-2})$ \\
$(8,8)$ & $2^{15}$ & $8000$ & $-8$ & $\Q(\sqrt{-2})$ \\
$(3,3)$ & $-2^4 3^3$ & $0$ & $-3$ & $\Q(\sqrt{-3})$ \\
$(42,-7)$ & $2^{12}7^3$ & $16581375$ & $-28$ & $\Q(\sqrt{-7})$ \\
$(21,112)$ & $-2^{12}7^3$ & $-3375$ & $-7$ & $\Q(\sqrt{-7})$ \\
\bottomrule
\end{tabular}
\end{table}

\begin{proposition}\label[proposition]{prop:central-local}
Let $(c,d)$ be a row of \cref{tab:central-cm}, with CM discriminant $D$.
Let $p\ge5$ be a prime satisfying $p\nmid\Delta(E_{c,d})$ and $\Leg Dp=-1$.  Then the vector
$\vec u=(u_k)_{0\le k\le p-1}$ defined by
\[
 u_0=0,
 \qquad u_k=\Leg kp\quad(1\le k\le p-1)
\]
is a kernel vector over $\Z$ of
$\mathcal M_p^{\mathrm{full}}(c,d)$.
\end{proposition}

\begin{proof}
Deuring's theorem (\cref{thm:deuring}) and \cref{rem:elliptic} together imply
\begin{equation}\label{eq:central-sum}
 t_p(E_{c,d})=-\sum_{x\in\F_p}\Leg{xq(x)}p=0,
\end{equation}
as an equality in $\Z$. For $j\ne0$, substituting $x=j/k$ yields 
\[ \sum_{k\in\F_p} \bigl(\mathcal M_p^{\mathrm{full}}(c,d)\bigr)_{j,k}u_k =\Leg jp\sum_{x\ne0}\Leg{xq(x)}p=0, \] 
where we have used the fact that $k=j/x$ implies $\Leg kp=\Leg jp\Leg{x^{-1}}p=\Leg jp\Leg xp$, and the terms corresponding to $x=0$ or $k=0$ naturally vanish. For $j=0$, the row-vector product evaluates to $\Leg dp\sum_{k\ne0}\Leg kp=0$. Thus, every row-vector product vanishes identically over $\Z$.
\end{proof}

We shall use the following elementary observation in the proofs of cases (a),
(c), and (d) of \cref{thm:main}.
\begin{itemize}
    \item If $n$ is odd, $\gcd(n,2D)=1$, and $\Leg Dn=-1$, then the factorization $\Leg Dn=\prod_{p^e\parallel n}\Leg Dp^e$ forces the existence of a prime divisor $p$ occurring with an odd exponent such that $\Leg Dp=-1$. Equivalently, this prime $p$ is inert in $\Q(\sqrt D)$.

    \item For the exceptional prime $p=3$ (which is excluded by the hypotheses of \cref{prop:central-local}), direct calculation yields the kernel vector
$\vec u_3=(0,1,-1)^{\mathsf T}$ over $\Z$  for each  of the pairs 
$(c,d)=(3,2),(42,-7),(21,112)$.
\end{itemize}

\begin{proof}[\textbf{Proof of \cref{thm:main}(a)}]
If $n$ is not squarefree, \cref{lem:composite}(i) applies.  Suppose that
$n$ is squarefree.  Since $n\equiv1\pmod4$ is not a sum of two squares, the two-squares theorem gives a prime $p\equiv3\pmod4$ dividing $n$.  For
$(c,d)=(3,2)$, the discriminant in \cref{tab:central-cm} is $D=-4$, so
$\Leg Dp=-1$.  If $p=3$, we use $\vec u_3$ as defined above. Otherwise \cref{prop:central-local}  provides a kernel vector for $\mathcal M_p^{\mathrm{full}}(c,d)$.  Here $n>p$ since $n\not\equiv p\pmod 4$. By
\cref{lem:composite}(ii), this kernel lifts to the truncated matrix $\mathcal M_n^{\mathrm {trun}}(3,2)$.
Thus $\{3,2\}_n=0$.
\end{proof}

\begin{proof}[\textbf{Proof of \cref{thm:main}(c)}]
If $n$ is not squarefree, \cref{lem:composite}(i) applies.  Assume that
$n$ is squarefree. 
Note that $\Leg{-2}n=-1$ when
$n\equiv5\pmod8$ and $\Leg{-3}n=-1$ when $n\equiv5\pmod {12}$.

\begin{itemize} 
\item For $(c,d)=(4,2)$ and $(8,8)$,  we have $D=-8$ according to
    \cref{tab:central-cm}.  As argued previously, this guarantees the existence of an inert prime $p$.  If $n>p$, \cref{prop:central-local} and
    \cref{lem:composite}(ii)  yield the desired result.  If $n=p$, then $p\equiv1\pmod4$ and
    $\Leg 2p=\Leg 8p=-1$. Hence, \cref{eq:central-sum} and \cref{lem:quadratic-boundary} establish  the same conclusion for each pair.
\item For $(c,d)=(3,3)$, we have $D=-3$.  The congruence
$n\equiv5\pmod {12}$ excludes $p=3$, allowing \cref{prop:central-local} and the inert prime argument to apply.  The composite case follows from
\cref{lem:composite}(ii), while the prime case is  resolved by
\cref{lem:quadratic-boundary} because $n\equiv1\pmod4$.
\end{itemize}
\end{proof}

\begin{proof}[\textbf{Proof of \cref{thm:main}(d)}]
If $n$ is not squarefree, \cref{lem:composite}(i) applies.  Assume that
$n$ is squarefree.  Note that in this case 
$\Leg{-7}n=\Leg n7=-1$.  The inert prime argument guarantees a prime divisor $p$ of $n$ that is inert in $\Q(\sqrt{-7})$. If $p=3$, we use $\vec u_3$ otherwise, we apply \cref{prop:central-local}.  If $n>p$, \cref{lem:composite}(ii) yields
the desired result.  If $n=p$, then $p\equiv1\pmod4$ and
$\Leg{-7}p=-1$ ensures that the condition $\Leg dp=-1$ holds for both  pairs
$(42,-7)$ and $(21,112)$. Thus \cref{lem:quadratic-boundary} applies to both.
\end{proof}

\subsection{The one-third coefficient}

The proof of case (b) of \cref{thm:main} relies on the coefficient with index $\lfloor p/3\rfloor$ in \cref{eq:q-expansion}. In this subsection, we specialize to $(c,d)=(2,2)$ and write
\[
 q(X)=X^2+2X+2,\qquad
 r=\lfloor \frac{p}{3}\rfloor,\qquad r'=p-1-r.
\]
We shall demonstrate that $a_r=a_{r'}=0$.

    Let $P_r(T)$ denote the $r$th Legendre polynomial, normalized by
\[
 (1-2Tz+z^2)^{-\frac{1}{2}}=\sum_{\ell\ge0}P_\ell(T)z^\ell.
\]
Choose $\alpha\in\overline{\F}_p$ with $\alpha^2=2$ and put
$t=-\alpha^{-1}$. We have 
\begin{equation}\label{eq:third-coeff}
 a_r=[X^r]q(X)^{\frac{p-1}{2}}=[X^r]\left[2\left(1-2t\frac{X}{\alpha}+\frac{X^2}{\alpha^2}\right)\right]^{\frac{p-1}{2}}=2^{\frac{p-1}{2}}\alpha^{-r}P_r(t)\quad\text{in }\overline{\F}_p.
\end{equation} 
Indeed, since $r<\frac{p-1}{2}<p$, the binomial expansion to the power of $\frac{p-1}{2}$ modulo $z^{r+1}$ relies only on the binomial coefficients $\binom{\frac{p-1}{2}}{\ell}$ for $0\le\ell\le r$. These are congruent to $\binom{-\frac{1}{2}}{\ell}$ in $\F_p$ because $\ell!$ is invertible in $\F_p$.

\begin{theorem}[{\cite[Theorem 3.1]{Sun2013}}]\label[theorem]{thm:legendre1}
 Let $p > 3$ be a prime and $t \in \Z_{(p)}$, the set of rational numbers with denominators coprime to $p$. Then
\[
P_{\left\lfloor \frac{p}{3} \right\rfloor}(t) \equiv -\left( \frac{p}{3} \right) \sum_{x=0}^{p-1} \left( \frac{x^3 + 3(4t - 5)x + 2(2t^2 - 14t + 11)}{p} \right) \pmod{p}.
\]
\end{theorem}

\begin{corollary}\label[corollary]{cor:hasse1}
    Denote $ g_T(x)=x^3+3(4T-5)x+2(2T^2-14T+11)$ and  $ \mathcal H_p^{(3)}(T)=[x^{p-1}]g_T(x)^{\frac{p-1}{2}}$, then
    \begin{equation}\label{eq:third-polynomial}
 P_r(T)=\Leg p3\mathcal H_p^{(3)}(T)
 \quad\text{in }\F_p[T].
\end{equation}
\end{corollary}
\begin{proof}
For every $u\in\F_p$, the preceding theorem says
\begin{equation}\label{eq:third-sun-pointwise}
 P_r(u)=-\Leg p3\sum_{x\in\F_p}\Leg{g_u(x)}p= \Leg p3 \mathcal H_p^{(3)}(u)
 \quad\text{in }\F_p,
\end{equation}
 since $\deg_x g_u(x)^{\frac{p-1}{2}}=\frac{3(p-1)}{2}<2(p-1)$.
A monomial contributing to $x^{p-1}$ in $g_T(x)^{\frac{p-1}{2}}$ arises from choosing $i, j, k$ factors from the cubic, linear, and constant terms, respectively. Thus, we have $i+j+k=\frac{p-1}{2}$ and $3i+j=p-1$. The degree of this monomial in $T$ is at most $j+2k=i$, and the condition $j\ge0$ implies $i\le\lfloor\frac{p-1}{3}\rfloor=r$. Hence, $\deg_T\mathcal H_p^{(3)}\le r<p$. Both sides of \cref{eq:third-sun-pointwise} are polynomials in $T$ of degree at most $r < p$. Since they agree on all $p$ elements of $\F_p$, their difference must be identically the zero polynomial in $\F_p[T]$, as desired.
\end{proof}

Consider the curve
\begin{equation}\label{eq:third-curve}
 E_T^{(3)}:\quad
 y^2=x^3+3(4T-5)x+2(2T^2-14T+11).
\end{equation}
If $T^2=\frac{1}{2}$, then
\begin{align*}
 &\Delta(E_T^{(3)})=-16\left\{4\left[3(4T-5)\right]^3+27\left[2(2T^2-14T+11)\right]^2\right\}=1728(4T+3)\ne0,\\
&j(E_T^{(3)})=1728\frac{4\left[3(4T-5)\right]^3}{4\left[3(4T-5)\right]^3+27\left[2(2T^2-14T+11)\right]^2}=2417472-3414528T,
\end{align*}
and $j(E_T^{(3)})$ is exactly a root of the Hilbert class polynomial (see \cref{lst:hilbert-class-polynomials})
\begin{equation*}
 H_{-24}(Y)=Y^2-4834944Y+14670139392.
\end{equation*}
Hence, $E_T^{(3)}$ is an elliptic curve over $\Q(T)$ with complex multiplication  by the order of discriminant $-24$ in $\Q(\sqrt{-6})$. Let $K=\Q(\sqrt2)$, choose $T_0=\pm\sqrt2/2$, and let $\mathfrak p$ be a prime ideal of $\mathcal{O}_K$ lying above $p$ such that the residue map sends $T_0$ to $t$. For $p>3$, the coefficients of $E_{T_0}^{(3)}$ are $\mathfrak p$-integral, while the discriminant $\Delta(E_{T_0}^{(3)})=1728(3\pm2\sqrt2)$ is a $\mathfrak p$-adic unit since its norm is $N_{K/\Q}(3\pm2\sqrt2)=1$. Hence, $E_{T_0}^{(3)}$ has good reduction at $\mathfrak p$, and its reduction modulo $\mathfrak p$ is precisely the curve $E_t^{(3)}$.

\begin{proof}[\textbf{Proof of \cref{thm:main}(b)}]

For $p\equiv13,17,19,23\pmod {24}$, we have $\Leg{-6}p=-1$. Therefore, $p$ is inert in $\Q(\sqrt{-6})$, and by Deuring's theorem (\cref{thm:deuring}), the reduced elliptic curve $E_t^{(3)}$ of \cref{eq:third-curve} is supersingular. Its Hasse invariant $\mathcal H_p^{(3)}(t)$ consequently vanishes. Equations \eqref{eq:third-polynomial} and \eqref{eq:third-coeff} yield $a_r=0$. Finally, \cref{lem:spectrum}, with $d=2$, gives $a_{r'}=0$.

\begin{itemize}
    \item Suppose first that $p\equiv17,23\pmod {24}$.  Then $p\equiv2\pmod3$ and
\[
 r=\frac{p-2}{3},\qquad r'=p-1-r=\frac{2p-1}{3}.
\]
Both $r$ and $r'$ are odd and thus, $\vec w_r-\vec w_{r'}$ lies in the kernel of the truncated matrix. Hence, we conclude $\{2,2\}_p\equiv0\pmod p$.

\item Assume now that $p\equiv13,19\pmod {24}$, and put
$N=(p-1)/3$.  Then $N$ is even and $r=N,r'=2N$.  Let
$C_0=(\F_p^\times)^3,C_1,C_2$ be the cubic cosets, and set
\[
 S_i=\sum_{x\in C_i}\Leg{q(x)}p\in\Z\qquad(0\le i\le2).
\] 
Fix a primitive cube root $\omega\in\F_p$ and label the cosets such that $x^N=\omega^i$ for $x\in C_i$.

The preceding argument gives $a_N=a_{2N}=0$.  Moreover,
$a_0=\Leg2p=-1$ and $a_{3N}=1$ in $\F_p$ gives
$a_0+a_{3N}=0$. Hence in $\F_p$,
\[
\begin{pmatrix}1&1&1\\1&\omega&\omega^2\\
1&\omega^2&\omega\end{pmatrix}
\begin{pmatrix}
    S_0\\S_1\\S_2
\end{pmatrix}
=
\begin{pmatrix}
    \sum_{x\ne0}q(x)^{\frac{p-1}{2}}\\
    \sum_{x\ne0}q(x)^{\frac{p-1}{2}}x^N\\
    \sum_{x\ne0}q(x)^{\frac{p-1}{2}}x^{2N}
\end{pmatrix}
=
\begin{pmatrix}
    -a_0-a_{3N}\\
    -a_{2N}\\
    -a_N
\end{pmatrix}
=0.
\]
Since $\omega\ne1$, we obtain $S_i\equiv0\pmod p$ for $0\le i\le2$.  As
$|S_i|\le|C_i|=(p-1)/3<p$, it follows that $S_0=S_1=S_2=0$ in $\Z$.
Since $N$ is even, both $\pm1\in C_0$.   For any $2\le j\le p-2$,
\[
 \sum_{k\in C_1}\Leg{j^2+2jk+2k^2}p
 =\sum_{x\in jC_1^{-1}}\Leg{q(x)}p=0,
\]
since $jC_1^{-1}$ is just one of the  cubic cosets. Therefore the indicator function  $v_k=\mathbf 1_{C_1}(k)$ restricted to $2\le k\le p-2$ provides a
nonzero kernel vector over $\Z$ for the truncated matrix, establishing that 
$\{2,2\}_p=0$.
\end{itemize}
\end{proof}
\begin{remark}
    In particular, $\{2,2\}_{23}=-9984540672=-2^{21}\cdot 3^2\cdot 23^2$  which verifies the difference between the two conclusions.
\end{remark}

\subsection{The one-quarter coefficient}

The proofs of cases (f), (g), and (h) of \cref{thm:main} use the coefficient
with index $\lfloor p/4\rfloor$ in \cref{eq:q-expansion}.  Let $p>3, p\nmid d$ and set
\[
 q(X)=X^2+cX+d,\qquad
 r=\left\lfloor\frac{p}{4}\right\rfloor,\qquad
 r'=p-1-r,
\]
We shall prove that $a_r=a_{r'}=0$ for the three pairs $(c,d)$ below.  Choose
$\alpha\in\overline{\F}_p$ with $\alpha^2=d$ and set $t=-c/(2\alpha)$, so that $t^2=c^2/(4d)$. By an argument analogous to the one preceding \cref{thm:legendre1}, we have
\begin{equation}\label{eq:quarter-coeff}
 a_r=[X^r]q(X)^{\frac{p-1}{2}}
 =d^{\frac{p-1}{2}}\alpha^{-r}
 [z^r](1-2tz+z^2)^{-\frac{1}{2}}
 =d^{\frac{p-1}{2}}\alpha^{-r}P_r(t)
 \quad\text{in }\overline{\F}_p.
\end{equation}

\begin{theorem}[{\cite[Theorem 2.1]{SunLegendreII}}]
Let $p > 3$ be a prime and $t \in \Z_{(p)}$, the set of rational numbers with denominators coprime to $p$. Then
   \[
   P_{\lfloor \frac{p}{4}\rfloor}(t) \equiv -\left(\frac{6}{p}\right) \sum_{x=0}^{p-1} \left(\frac{x^3 - \frac{3}{2}(3t+5)x + 9t + 7}{p}\right) \pmod p. 
   \] 
\end{theorem}

\begin{corollary}
 Denote $ f_T(x)=x^3-\frac{3}{2}(3T+5)x+9T+7$ and
 $\mathcal H_p^{(4)}(T)=[x^{p-1}]f_T(x)^{\frac{p-1}{2}}$.  Then
\begin{equation}\label{eq:quarter-polynomial}
 P_r(T)=\Leg6p\mathcal H_p^{(4)}(T)
 \quad\text{in }\F_p[T].
\end{equation}
\end{corollary}

\begin{proof}
The proof is completely analogous to that of \cref{cor:hasse1}.
\end{proof}
Consider the elliptic curve
\begin{equation*}
 E_T^{(4)}:\quad y^2=x^3-\frac{3}{2}(3T+5)x+9T+7.
\end{equation*}
Its discriminant and $j$-invariant are
\begin{align*}
 \Delta(E_T^{(4)})
 &=-16\left\{4\left[-\frac32(3T+5)\right]^3
 +27(9T+7)^2\right\}=5832(T-1)^2(T+1),\\
j(E_T^{(4)})
 &=1728\frac{4\left[-\frac32(3T+5)\right]^3}
 {4\left[-\frac32(3T+5)\right]^3+27(9T+7)^2}
=64\frac{(3T+5)^3}{(T-1)^2(T+1)}.
\end{align*}
The values needed in this subsection are displayed in \cref{tab:quarter-cm} (see \cref{lst:quarter-cm}).
\begin{table}[ht]
\centering
\caption{Data used for the one-quarter coefficient $a_{\left\lfloor\frac{p}{4}\right\rfloor}$.}
\label{tab:quarter-cm}
\small
\renewcommand{\arraystretch}{1.15}
\resizebox{\textwidth}{!}{%
\begin{tabular}{@{}c c c c c c c@{}}
\toprule
\multicolumn{7}{c}{$E_t^{(4)}:\ y^2=x^3-\frac{3}{2}(3t+5)x+9t+7$}\\
\midrule
$(c,d)$
& parameter $t^2=\frac{c^2}{4d}$
& $j(E_t^{(4)})$
& $\Delta_t$
& $D$
& $H_{D}(Y)$
& CM field
\\
\midrule
$(5,5)$
& $t^2=\frac{5}{4}$
& $632000+565760t$
& $1458(t-1)$
& $-20$
& $Y^2-1264000Y-681472000$
& $\Q(\sqrt{-5})$
\\
$(10,9)$
& $t=-\frac{5}{3}$
& $0$
& $-2^{10}3^3$
& $-3$
& $Y$
& $\Q(\sqrt{-3})$
\\
$(8,18)$
& $t^2=\frac{8}{9}$
& $2417472+2560896t$
& $648(1-t)$
& $-24$
& $Y^2-4834944Y+14670139392$
& $\Q(\sqrt{-6})$
\\
\bottomrule
\end{tabular}%
}
\end{table}

Thus, in these cases $E_T^{(4)}$ admits CM by the corresponding order. To verify good reduction, observe that for the two quadratic parameter cases, the norm
\[
N_{\Q(T)/\Q}(T\pm 1) = 1-T^2
\]
evaluates to $-1/4$ and $1/9$, respectively. Hence, for every prime $p>3$, both $T+1$ and $T-1$ are $\mathfrak p$-adic units at any prime ideal $\mathfrak p$ lying above $p$. Since the coefficients of $E_T^{(4)}$ are $\mathfrak p$-integral and the discriminant $\Delta(E_T^{(4)})=2^3 3^6(T-1)^2(T+1)$ is a $\mathfrak p$-adic unit, the curve exhibits good reduction at every prime ideal $\mathfrak p$ lying above $p>3$. For $T=-5/3$, the only prime factors occurring in the denominators of $T, T\pm1$ are $2$ and $3$. Therefore, the same conclusion holds for any $p>3$. By choosing a prime ideal above $p$ whose residue map sends $T$ to $t$, we see that the reduction of $E_T^{(4)}$ is exactly $E_t^{(4)}$.

\begin{proposition}
\label[proposition]{prop:quarter-zero}
For any row of \cref{tab:quarter-cm}, let $K$ be the CM field in the table and
let $p>3$ be inert in $K$.  Then
\[
 a_r=a_{p-1-r}=0.
\]

\end{proposition}

\begin{proof}
By Deuring's theorem (\cref{thm:deuring}), \cref{eq:quarter-polynomial}, and \cref{eq:quarter-coeff}, we conclude that $a_r=0$. Consequently, \cref{lem:spectrum} implies $a_{p-1-r}=0$.
\end{proof}

Observe that for $r'=p-1-r$, we always have $r\equiv r'\pmod2$ for every odd prime $p$. Consequently, \cref{prop:boundary} is always applicable whenever the hypothesis of \cref{prop:quarter-zero} is satisfied.

\begin{remark}\label[remark]{rem:quartic-coset-vector}
    When $p\equiv1\pmod4$, put $N=\frac{p-1}{4}$, choose a primitive root
$g$, and write $C_i=g^i(\F_p^\times)^4$.  Let
$\iota=g^N$, so that $\iota^2=-1$, and define $\vec v=(v_k)_{k\in\F_p}$
by
\begin{equation*}
 v_0=0,\qquad
 v_k=
 \begin{cases}
  1,&k\in C_1,\\
  -1,&k\in C_3,\\
  0,&k\in C_0\cup C_2
 \end{cases}
 \quad(k\in\F_p^\times).
\end{equation*}
For $k\in C_i$ one has $k^N=\iota^i$, and hence $ v_k=(2\iota)^{-1}\bigl(k^N-k^{3N}\bigr).$ Thus, the restriction of $\vec v$ to $\F_p^\times$ is a nonzero scalar multiple of $\vec w_N-\vec w_{3N}$ over $\F_p$. If $a_N=a_{3N}=0$, then every nonzero row product with $\vec v$ is divisible by $p$. Each such product has at most $2N=\frac{p-1}{2}$ nonzero summands, each taking the value $1$ or $-1$. Since $\frac{p-1}{2} < p$, divisibility by $p$ forces the product to be identically zero in $\Z$. The product corresponding to the zeroth row evaluates to zero simply because $|C_1|=|C_3|$. Finally, since $1\in C_0$ and $-1\in C_0\cup C_2$, we have $v_0=v_1=v_{-1}=0$. We therefore conclude that $\vec v$ is a kernel vector over $\Z$ of the full matrix, and its restriction constitutes a kernel vector over $\Z$ of the truncated matrix.
\end{remark}

\begin{proof}[\textbf{Proof of \cref{thm:main}(f)}]
If $n$ is not squarefree, \cref{lem:composite}(i) establishes  the result.  Suppose
that $n$ is squarefree. 
\begin{itemize}
    \item Since $n\equiv 13,17\pmod {20}$ is a sum of two squares, every prime
divisor of $n$ is $1\pmod4$. Note
$\Leg{-5}n=-1$,  multiplicativity therefore
gives a prime $p\mid n$ with $\Leg{-5}p=-1$.  This prime is inert in
$\Q(\sqrt{-5})$ and satisfies $p\equiv1\pmod4$.  The $(5,5)$ row of
\cref{tab:quarter-cm} and \cref{rem:quartic-coset-vector} above give a kernel vector over
$\Z$ of  $\mathcal M_p^{\mathrm{trun}}(5,5)$.  If $n>p$, \cref{lem:composite}(ii)
transports this vector to the truncated matrix modulo $n$. If $n=p$, the
same vector already vanishes at $0,1,-1$.  Hence $\{5,5\}_n=0$.
    \item Suppose that $n\equiv11,19\pmod {20}$, then
$\Leg{-5}n=-1$, so there is a prime $p\mid n$ with
$\Leg{-5}p=-1$.  By \cref{prop:quarter-zero},
$a_r=a_{r'}=0$. If $p\equiv1\pmod4$, then \cref{rem:quartic-coset-vector} gives a kernel vector over
$\Z$ of  $\mathcal M_p^{\mathrm{trun}}(5,5)$.  Thus $\{5,5\}_n=0$ if $n=p$, and
\cref{lem:composite}(ii) gives $\{5,5\}_n=0$ if $n>p$.
If $p\equiv3\pmod4$, then $r\equiv r' \pmod 2$, hence
$\vec w_r-\vec w_{r'}$ vanishes at $1$ and $-1$, and
\cref{prop:boundary} gives
$\{5,5\}_p\equiv0\pmod p$ when $n=p$.  If $n>p$,
$\vec w_r-\vec w_{r'}$, extended by zero at the zero coordinate, is a kernel vector of  $\mathcal M_p^{\mathrm{full}}(5,5)$ modulo $p$. Therefore
\cref{lem:composite}(ii) gives
$\{5,5\}_n\equiv0\pmod p$.

In all cases, some prime divisor $p$ of $n$ divides $\{5,5\}_n$.  Consequently $\Leg{\{5,5\}_n}n=0.$

\end{itemize}

\end{proof}

\begin{proof}[\textbf{Proof of \cref{thm:main}(g)}]
\begin{itemize}[leftmargin=2em]
\item Suppose that $n\equiv5\pmod {12}$ is a sum of two squares.  If $n$ is not squarefree, \cref{lem:composite}(i) applies.  Otherwise every prime divisor of $n$ is $1\pmod4$, while
$\Leg{-3}n=-1$ in this case.  Hence some prime $p\mid n$ is inert in $\Q(\sqrt{-3})$ and satisfies $p\equiv1\pmod4$.  The $(10,9)$
row of \cref{tab:quarter-cm} and \cref{rem:quartic-coset-vector} give
$\{10,9\}_n=0$.

\item Let $p\equiv11\pmod {12}$.  Since
$\Leg{-3}p=-1$, the prime $p$ is inert in $\Q(\sqrt{-3})$.
Hence \cref{prop:quarter-zero}, applied to $(c,d)=(10,9)$, implies that the
restriction of $\vec w_r-\vec w_{r'}$ is a nonzero kernel vector of
$\mathcal M^{\mathrm{trun}}_p(10,9)$ modulo $p$, since
$r\equiv r'\pmod2$.
Thus $ \{10,9\}_p\equiv0\pmod p$.

\end{itemize}
\end{proof}

\begin{proof}[\textbf{Proof of \cref{thm:main}(h)}]
\begin{itemize}[leftmargin=2em]
\item Suppose that $n\equiv13,17\pmod {24}$ is a sum of two squares.  If
$n$ is not squarefree, \cref{lem:composite}(i) gives the result.  Otherwise every prime divisor of $n$ is $1\pmod4$, and
$\Leg{-6}n=-1$ in this case.  Thus some prime $p\mid n$
is inert in $\Q(\sqrt{-6})$ and satisfies $p\equiv1\pmod4$.  The $(8,18)$
row of \cref{tab:quarter-cm} and \cref{rem:quartic-coset-vector} yield
$\{8,18\}_n=0$.

\item Let $p\equiv23\pmod {24}$. Since $r\equiv r'\pmod 2$, the vector $\vec w_r-\vec w_{r'}$ with 
\cref{prop:boundary} gives $\{8,18\}_p\equiv0\pmod p$.

\item Let $p\equiv19\pmod {24}$. The vectors $\vec w_r$ and $\vec w_{r'}$ have zero eigenvalue by \cref{prop:quarter-zero}, and both of their exponents are even. The constant vector also has zero eigenvalue because $a_0+a_{p-1}=\Leg{18}p+1=0$. Since these three linearly independent vectors take the same value at $1$ and $-1$, the evaluation map of \cref{prop:boundary} has rank one on their span. Thus, $\mathcal M_p^{\mathrm{trun}}(8,18)$ has nullity at least two over $\F_p$, which implies that $p^2\mid\{8,18\}_p$.
\end{itemize}
\end{proof}

\subsection{Low-degree coefficients and full nonzero-residue determinants}

The proofs of case (e) of \cref{thm:main} and \cref{thm:oldSUN}  rely on the coefficients of degrees two and three in \cref{eq:q-expansion}. Assuming $p\nmid d$, a direct extraction yields
\begin{align*}
 a_2&=[X^2]q(X)^{\frac{p-1}{2}}
 =\frac{p-1}{2}d^{\frac{p-3}{2}}
 +\binom{\frac{p-1}{2}}{2}c^2d^{\frac{p-5}{2}}=\frac{p-1}{2}d^{\frac{p-5}{2}}
 \left(d+\frac{p-3}{4}c^2\right),\\
 a_3&=[X^3]q(X)^{\frac{p-1}{2}}
 =\frac{(p-1)(p-3)}{4}cd^{\frac{p-5}{2}}
 +\binom{\frac{p-1}{2}}{3}c^3d^{\frac{p-7}{2}}=\frac{(p-1)(p-3)}{4}cd^{\frac{p-7}{2}}
 \left(d+\frac{p-5}{12}c^2\right).
\end{align*}
In $\F_p$, $4d=3c^2$ implies $a_2=0$, while $12d=5c^2$ implies $a_3=0$.
\cref{lem:spectrum} then gives $a_{p-3}=0$ and $a_{p-4}=0$,
respectively.  The primes dividing $d$ are treated separately below.

\begin{proof}[\textbf{Proof of \cref{thm:main}(e)}]
\begin{itemize}
\item For $(c,d)=(2,3)$ and a prime $p\ge7$, the exponents $2$ and $p-3$
are distinct even indices.  Hence \cref{prop:boundary} gives
\begin{equation}\label{eq:23-prime-p}
 p\mid \{2,3\}_p.
\end{equation}
For $p=5$ the truncated matrix is $\begin{pmatrix}1&-1\\-1&1\end{pmatrix}$,
and its determinant is zero.  This proves \cref{eq:23-prime-p} for every
prime $p>3$.

If $p\equiv5,7\pmod {12}$, then $\Leg3p=-1$, and therefore $\vec w_0$ has eigenvalue $-(a_0+a_{p-1})=-\left(\Leg3p+1\right)=0.$ For $p\ge7$, the three independent vectors
$\vec w_0,\vec w_2,\vec w_{p-3}$ lie in 
$\ker_{\F_p}\mathcal M_p^\times(2,3)$.  All take the value $1$
at both $1$ and $-1$.  The evaluation map of \cref{prop:boundary} therefore
has rank one on their span, so the truncated nullity is at least two. Thus 
\begin{equation*}
 p^2\mid\{2,3\}_p
 \qquad(p\equiv5,7\pmod {12},\ p\ge7).
\end{equation*}
The missing prime $p=5$ has determinant zero, as shown above.

\item For $(c,d)=(6,15)$ and every prime $p>7$, the identity $12d=5c^2$ gives zero eigenvalues at the distinct odd indices
$3$ and $p-4$.  Their difference vanishes at $1$ and $-1$, so $p\mid \{6,15\}_p.$
In general this is only a one-dimensional boundary-compatible kernel: for
example, direct calculation gives $\{6,15\}_11=77$.

\end{itemize}
\begin{itemize}
\item \emph{Composite moduli.}
By \cref{lem:composite}(i), a non-squarefree modulus makes the relevant
truncated determinant zero.  We may therefore assume throughout this item
that $n>3$ is squarefree and composite.  For $(2,3)$, let $ \nu_p(c,d)=\dim_{\F_p}\ker\mathcal M_p^{\mathrm{full}}(c,d).$
Part (iii) of \cref{lem:composite}, with the three deleted indices
$0,1,-1$, gives the uniform estimate
\begin{equation}\label{eq:global-nullity}
  \dim_{\F_p}\ker \mathcal M_n^{\mathrm{trun}}(c,d)\ge \nu_p(c,d)\frac{n}{p}-3.
\end{equation}

The following table gives the required lower bounds.  For an exceptional prime, the final column gives a full matrix and its kernel vector or vectors. All coordinates are indexed by $0,1,\ldots,p-1$.  Every entry of a displayed
matrix is a Legendre symbol modulo the prime in its row.  The bounds need not be exact.
\begin{equation*}
\renewcommand{\arraystretch}{1.35}
\begin{array}{c|c|c|c}
(c,d)&p&\nu_p(c,d)\ge&
\mathcal M_p^{\mathrm{full}}(c,d)\text{ and some kernels in } \F_p\\
\hline
(2,3)&3&1&
\scriptstyle\begin{gathered}
\left(\begin{smallmatrix}0&0&0\\1&0&-1\\1&-1&0\end{smallmatrix}\right),\,\, 
\vec\kappa_{2,3;3}=\left(\begin{smallmatrix}1\\1\\1\end{smallmatrix}\right)
\end{gathered}\\
(2,3)&5&1&
\scriptstyle\begin{gathered}
\left(\begin{smallmatrix}0&-1&-1&-1&-1\\1&1&-1&1&-1\\
1&1&1&-1&-1\\1&-1&-1&1&1\\1&-1&1&-1&1\end{smallmatrix}\right),\,\,
\vec\kappa_{2,3;5}=\left(\begin{smallmatrix}0\\1\\-1\\-
  1\\1\end{smallmatrix}\right)
\end{gathered}\\
(2,3)&p\ge7&2&\vec w_2,\ \vec w_{p-3}\\
(6,15)&3&2&
\scriptstyle\begin{gathered}
\left(\begin{smallmatrix}0&0&0\\1&1&1\\1&1&1\end{smallmatrix}\right),\,\,
\vec\kappa_{6,15;3}=\left(\begin{smallmatrix}2\\1\\0\end{smallmatrix}\right),\,\,
\vec\kappa'_{6,15;3}=\left(\begin{smallmatrix}2\\0\\1\end{smallmatrix}\right)
\end{gathered}\\
(6,15)&5&2&
\scriptstyle\begin{gathered}
\left(\begin{smallmatrix}0&0&0&0&0\\1&-1&-1&1&0\\
1&1&-1&0&-1\\1&-1&0&-1&1\\1&0&1&-1&-1\end{smallmatrix}\right),\,\,
\vec\kappa_{6,15;5}=\left(\begin{smallmatrix}4\\3\\2\\1\\0\end{smallmatrix}\right),\,\,
\vec\kappa'_{6,15;5}=\left(\begin{smallmatrix}2\\3\\4\\0\\1\end{smallmatrix}\right)
\end{gathered}\\
(6,15)&7&1&
\scriptstyle\begin{gathered}
\left(\begin{smallmatrix}0&1&1&1&1&1&1\\1&1&-1&0&-1&0&-1\\
1&-1&1&0&-1&-1&0\\1&0&0&1&-1&-1&-1\\
1&-1&-1&-1&1&0&0\\1&0&-1&-1&0&1&-1\\
1&-1&0&-1&0&-1&1\end{smallmatrix}\right),\,\,
\vec\kappa_{6,15;7}=\left(\begin{smallmatrix}0\\-1\\-1\\1\\-1\\1\\1\end{smallmatrix}\right)
\end{gathered}\\
(6,15)&p\ge11&2&\vec w_3,\ \vec w_{p-4}
\end{array}
\end{equation*}
Direct multiplication gives zero in each exceptional row, and the two
displayed vectors for $p=3,5$ in the $(6,15)$ rows are independent.  The
remaining rows follow from the two distinct power-function vectors.

Let $n>3$ be squarefree and composite. For $(2,3)$, if $p\mid n$ and $p\notin\{3,5\}$, then $\frac{n}{p}\ge3$, meaning the lower bound in \cref{eq:global-nullity} is at least $3$. For $p=3$, it is at least $5-3=2$, since $\frac{n}{3}\ge5$. For $p=5$, either $n=15$ or $\frac{n}{5}\ge7$. In the latter case, the lower bound is at least $7-3=4$. The exceptional determinant is exactly zero: with coordinates indexed by $2,3,\ldots,13$, the nonzero integer vector
\[
  \vec\kappa_{15}=(1,-1,0,0,1,1,0,1,0,0,0,0)
\]
is a kernel vector of the truncated matrix. Thus, $\{2,3\}_{15}=0$. For every other squarefree composite $n$ and every $p\mid n$, the nullity of the truncated matrix over $\F_p$ is at least two, which yields $p^2\mid\{2,3\}_n$ for every $p\mid n$.  Since $n$
is squarefree,
\begin{equation*}
 n^2\mid\{2,3\}_n
 \qquad(n>3\text{ odd and composite}).
\end{equation*}
For $(2,3)$, the prime argument proves $p^2\mid\{2,3\}_p$ for
$p\equiv5,7\pmod {12}$.  Since an odd integer is
congruent to one of $1,3,5,7,9,11$ modulo $12$, the only prime classes not
covered are $1$ and $11$, namely
$p\equiv\pm1\pmod {12}$.  Thus
\[
 n^2\mid\{2,3\}_n\qquad
 \bigl(n>3,\ n\not\equiv\pm1\pmod {12}\bigr),
\]
which is precisely the additional assertion in Conjecture 5.5(i). The composite argument gives the stated stronger result.  
\item For $(c,d)=(6,15)$, the cases $p=3,5$ give lower bounds at least $2\cdot5-3$
and $2\cdot3-3$, respectively.  If $p\ge11$, the bound is at least
$2\cdot3-3$.  For $p=7$, it is at least $5-3$ unless $n=21$. In that exceptional case, the vector
\[
  \vec\kappa_{21}=(1,-1,1,-1,-1,0,1,0,0,0,0,0,0,0,0,0,0,0),
\]
indexed by $2,3,\ldots,19$, is a kernel vector over $\Z$.  Hence
$p\mid\{6,15\}_n$ for every $p\mid n$, and squarefreeness gives
$n\mid\{6,15\}_n$.

\end{itemize}
\end{proof}

\begin{proof}[\textbf{Proof of \cref{thm:oldSUN}}]
If $c=0$, replacing a retained row index $j$ with $-j$ leaves every entry unchanged, so the matrix must contain two identical rows. If $4d=3c^2$ in $\F_p$, then $a_2=a_{p-3}=0$. For $p\ge7$, the corresponding eigenvectors are distinct. Likewise, if $12d=5c^2$, then $a_3=a_{p-4}=0$ and these indices are distinct for $p\ge11$. Therefore, by \cref{lem:spectrum}, the full matrix defining $(c,d)_p$ has nullity at least two over $\F_p$, which implies that $(c,d)_p$ is divisible by $p^2$.
\begin{itemize}[leftmargin=2em]
\item Let $n=p$ be prime.  For $(c,d)=(2,3)$ and $p\ge7$,
The above discussion gives $p^2\mid(2,3)_p$.  For $p=5$, the
matrix $\mathcal M^{\mathrm{full}}_5(2,3)$ displayed above annihilates
\[
  \vec\kappa_{2,3;5}=(0,1,-1,-1,1)^{\mathsf T}
\]
over $\Z$, so $(2,3)_5=0$.  Likewise, for $(c,d)=(6,15)$ and
$p\ge11$, The above discussion gives $p^2\mid(6,15)_p$, while
the displayed matrix $\mathcal M^{\mathrm{full}}_7(6,15)$ annihilates
\[
  \vec\kappa_{6,15;7}=(0,-1,-1,1,-1,1,1)^{\mathsf T}
\]
over $\Z$, so $(6,15)_7=0$.

\item \emph{Composite moduli.} By \cref{lem:composite}(i), a non-squarefree modulus causes the full determinant to vanish. We may therefore assume that $n$ is squarefree and composite. Fix a prime factor $p\mid n$ and set $M=\frac{n}{p}$. For $(2,3)$, the local full matrix is singular
over $\F_p$: if $p\ge5$, extend $\vec w_2$ by zero at the zero coordinate,
as in \cref{lem:spectrum}. If $p=3$, use
\[
 \vec\kappa_{2,3;3}=(1,1,1)^{\mathsf T}.
\]
For $(6,15)$, use the extension of $\vec w_3$ when $p\ge7$, and use
\[
  \vec\kappa_{6,15;3}=(2,1,0)^{\mathsf T},
  \qquad
  \vec\kappa_{6,15;5}=(4,3,2,1,0)^{\mathsf T},
\]
when $p=3,5$, respectively.  Thus in either case the local nullity is at
least one.

Part (iii) of \cref{lem:composite}, with the single deleted index $0$, establishes that the nullity is at least $M-1\ge2$ for the matrix at modulus $n$. Consequently, $p^2$ divides the corresponding determinant for every $p\mid n$.  Since $n$ is
squarefree, this gives
\[
 n^2\mid(2,3)_n,\qquad n^2\mid(6,15)_n.
\]

\item The prime and squarefree composite cases above, together with the repeated-row observation from \cref{lem:composite}(i), complete the proof.
\end{itemize}
\end{proof}

\begin{remark}\label[remark]{rem:gene}
    The case studies treated in the preceding four subsections are not  isolated phenomena, but instances of a flexible generating mechanism.  By varying the coefficient index in the expansion   of $(X^2+cX+d)^{\frac{p-1}{2}}$ and applying the same three-step reduction (specialize parameters $\to$ identify   the CM elliptic curve $\to$ apply Deuring's theorem), one obtains many further families of $(c,d)$ pairs whose  associated determinants exhibit prescribed vanishing or  divisibility behaviour.

    As a representative example, the one-sixth coefficient rests  on \cite[Theorem 2.2]{ZHSunIII}, which states that for a
    prime $p>3$ and $m,n\in \Z_{(p)}$ with $m\not\equiv0\pmod{p}$,
    \[
        P_{\lfloor \frac{p}{6}\rfloor}\Bigl(\frac{n}{2m^3}\Bigr)
        \equiv -\Leg{3m}{p}
        \sum_{x=0}^{p-1}
        \Leg{x^3-3m^2x+n}{p}
        \pmod{p}.
    \]
    Following the approach used for the one-third and one-quarter  coefficients, one has
    \begin{equation*}
        P_{\lfloor \frac{p}{6}\rfloor}(T)
        =\Leg3p\mathcal H_p^{(6)}(T)
        \quad\text{in }\F_p[T],
    \end{equation*}
    where $p>5$, $F_T(x)=x^3-3x+2T$ and  $\mathcal H_p^{(6)}(T)=[x^{p-1}]F_T(x)^{\frac{p-1}{2}}$.  Therefore, choosing $(c,d)$ and $p$ to satisfy the conditions  of \cref{thm:deuring}, it can be shown that:
    \begin{enumerate}
        \item If $p>5$ and $p\equiv5,11\pmod{12}$, then
        $p\mid\{22,125\}_p$.  If $p\equiv17,53\pmod{60}$, then
        in fact $p^2\mid\{22,125\}_p$.

        \item If $p>5$ and $p\equiv7,13\pmod{24}$, then
        $\{28,250\}_p=0$.  If $p\equiv5,23\pmod{24}$, then
        $p\mid\{28,250\}_p$.  In the subcases
        $p\equiv29,101\pmod{120}$, one has
        $p^2\mid\{28,250\}_p$.

        \item If $p>7$ and
        $p\equiv13,19,31\pmod{42}$, then
        $\{126,2625\}_p=0$.  If
        $p\equiv5,17,41\pmod{42}$, then
        $p\mid\{126,2625\}_p$.
    \end{enumerate}
    The proofs are structurally identical to those of  \cref{thm:main} and thus are omitted here.   The composite modulus case follows similarly via  \cref{lem:composite}.  
    
    We expect that systematically enumerating admissible $(c,d)$ pairs for different coefficients will uncover further identities of the same shape. Moreover, one may consider determinants   associated to polynomials of higher degree in place of the quadratic form $X^2+cX+d$. The reduction to Hasse invariants would then involve Jacobians of higher-genus curves or more general abelian varieties, where \cite{Zaytsev} provides the necessary generalization of Deuring's theorem.
\end{remark}

\appendix

\section{Exact SageMath Verification of the Used CM Data}
\label{app:cm-verification}

All calculations in the following listings are exact.  The discriminant in the output is that of the displayed
Weierstrass equation.  Each listing can be run by pasting it directly into \url{https://sagecell.sagemath.org/} and selecting \texttt{Evaluate}. For reliable reproduction, copy the listings from this appendix's \TeX \, source file rather than from the rendered PDF.

\begin{lstlisting}[language=Python,
caption={SageMath verification of \cref{tab:central-cm}.},
label={lst:central-cm}]
from sage.schemes.elliptic_curves.cm import hilbert_class_polynomial

rows = [(3, 2), (4, 2), (8, 8), (3, 3), (42, -7), (21, 112)]

def cm_discriminant(j):
    """Find D in [-50,-3] such that H_D(j)=0."""
    candidates = [
        D for D in range(-3, -51, -1)
        if D % 4 in (0, 1) and hilbert_class_polynomial(D)(j) == 0
    ]
    assert len(candidates) == 1, (j, candidates)
    return candidates[0]

def cm_field(D):
    radicand = -prod(p for p, e in factor(-D) if e % 2)
    if radicand == -1:
        return r"\mathbb Q(i)"
    return r"\mathbb Q(\sqrt{%s})" % radicand

print(r"$(c,d)$ & $\Delta(E_{c,d})$ & "
      r"$j(E_{c,d})$ & $D$ & CM field \\")
for c, d in rows:
    E = EllipticCurve([0, c, 0, d, 0])
    Delta, j = E.discriminant(), E.j_invariant()
    D = cm_discriminant(j)
    print(r"$(%s,%s)$ & $%s$ & $%s$ & $%s$ & $%s$ \\" % (
        c, d, latex(factor(Delta)), latex(j), D, cm_field(D)
    ))
\end{lstlisting}

\begin{lstlisting}[language=Python,
caption={Hilbert class polynomials in the range $-30\leq D\leq -3$.},
label={lst:hilbert-class-polynomials}]
from sage.schemes.elliptic_curves.cm import hilbert_class_polynomial

R = PolynomialRing(ZZ, "Y")
for D in range(-30, -2):
    if D % 4 in (0, 1):
        H = R(hilbert_class_polynomial(D))
        print(r"$D=%s$: $H_D(Y)=%s$ \\" % (D, latex(H)))
\end{lstlisting}

\begin{lstlisting}[language=Python,
caption={SageMath verification of \cref{tab:quarter-cm}.},
label={lst:quarter-cm}]
from sage.schemes.elliptic_curves.cm import hilbert_class_polynomial

P = PolynomialRing(QQ, "X"); X = P.gen()
R = PolynomialRing(ZZ, "Y")
rows = [(5, 5, QQ(5)/4, True), (10, 9, -QQ(5)/3, False),
        (8, 18, QQ(8)/9, True)]

def cm_discriminant(j):
    candidates = [D for D in range(-3, -51, -1)
                  if D % 4 in (0, 1) and hilbert_class_polynomial(D)(j) == 0]
    assert len(candidates) == 1, (j, candidates)
    return candidates[0]

def cm_field(D):
    radicand = -prod(p for p, e in factor(-D) if e % 2)
    return r"\mathbb Q(i)" if radicand == -1 else r"\mathbb Q(\sqrt{%s})" % radicand

print(r"$(c,d)$ & parameter & $j$ & $\Delta_t$ & $D$ & "
      r"$H_D(Y)$ & CM field \\")
for c, d, q, quadratic in rows:
    if quadratic:
        K = NumberField(X^2 - q, "t"); t = K.gen()
        parameter = r"t^2=%s" % latex(q)
    else:
        K, t = QQ, q
        parameter = r"t=%s" % latex(t)
    E = EllipticCurve(K, [0, 0, 0, -K(3)/2*(3*t+5), 9*t+7])
    Delta, j = E.discriminant(), E.j_invariant()
    D = cm_discriminant(j)
    H = R(hilbert_class_polynomial(D))
    print(r"$(%s,%s)$ & $%s$ & $%s$ & $%s$ & $%s$ & $%s$ & $%s$ \\" % (
        c, d, parameter, latex(j), latex(Delta), D, latex(H), cm_field(D)
    ))
\end{lstlisting}

\end{document}